\documentclass{amsart}
\usepackage{amsmath,amssymb,bbm,color,url}
\usepackage{hyperref}
\usepackage[active]{srcltx}

\newcommand{\mathd}{\mathrm{d}}
\newcommand{\tmop}[1]{\ensuremath{\operatorname{#1}}}
\newcommand{\tmtextit}[1]{{\itshape{#1}}}

\definecolor{grey}{rgb}{0.75,0.75,0.75}
\definecolor{orange}{rgb}{1.0,0.5,0.5}
\definecolor{brown}{rgb}{0.5,0.25,0.0}
\definecolor{pink}{rgb}{1.0,0.5,0.5}
\newtheorem{corollary}{Corollary}

\newtheorem{lemma}{Lemma}
\newtheorem{proposition}{Proposition}
\newtheorem{theorem}{Theorem}

\newcommand{\ideal}[1]{\ensuremath{\langle #1 \rangle}}

\newcommand{\nats}{\mathbb{N}}
\newcommand{\ints}{\mathbb{Z}}
\newcommand{\reals}{\mathbb{R}}
\newcommand{\compls}{\mathbb{C}}
\newcommand{\MFn}{\tmop{MFn}}

\newcommand{\gothic}{\mathfrak}
\newcommand{\gm}{{\gothic m}}
\begin{document}

\title{Constructive Theory of Banach algebras}
\author{Thierry Coquand}
\address{Thierry Coquand\\
Computing Science Department, G\"oteborg University}
\author{Bas Spitters}
\address{Bas Spitters\\
Computer Science Department, Radboud University Nijmegen}
\date{\today}
\begin{abstract}
We present a way to organize a constructive development of the theory
of Banach algebras, inspired by works of Cohen, de Bruijn and Bishop
~\cite{Cohen,deBruijn,Bishop/Bridges:1985,Coquand/Stolzenberg}. We illustrate
this by giving elementary proofs of Wiener's result on the inverse
of Fourier series and Wiener's Tauberian Theorem, in a sequel to this paper
we show how this can be used in a localic, or point-free,
description of the spectrum of a Banach algebra.
\end{abstract}
\maketitle
\section*{Introduction}

The applications of the theory of Banach algebras to ``concrete'' theorems in
analysis,
such as Wiener's theorem on the inverse of Fourier series, or Wiener's Tauberian
Theorem~\cite{Wiener},
constitute a striking example of the power of abstract methods in mathematics.
The abstract argument is
short and easy to grasp, when compared to Wiener's explicit constructions.
Furthermore, it is
highly non constructive and uses Zorn's Lemma. As such, it is a perfect illustration of
Hilbert's defense of the use of the law of excluded-middle against Brouwer~\cite{Hilbert}. A natural
question is whether the abstract argument does not contain, in some implicit
way, an actual
construction.
This question has been analysed and answered by P. Cohen~\cite{Cohen}.
Closely related are later works by de Bruijn and van der Meiden~\cite{deBruijn},
and by Bishop and Bridges~\cite{Bishop/Bridges:1985}. Here we present a slightly different
analysis of the non-constructive argument, thus providing an elementary treatment.
Furthermore it allows us to reformulate some key results of Bishop and
Bridges~\cite{Bishop/Bridges:1985} in a localic, or point-free, setting. This will be done in a
sequel to the present paper.

This paper is organised as follows. The first section explains informally the
main idea of the paper. The next sections, until Section~\ref{Tauberian}, are presented in
elementary Bishop style mathematics, in a detailed way so that they can hopefully be
readily implemented in type theory using the framework presented in~\cite{Riemann}.
This presentation follows ideas presented in~\cite{Bridger/Stolzenberg}.

\subsection*{Notation}
We use the following conventions, unless otherwise indicated:\\
$A,B$ for Banach algebras.\\
$a,b,c,\ldots$ for elements of a Banach algebra.\\
$\varphi,\psi$ for functionals.\\
$\Gamma_R$ denotes the circle with radius $R$, $\Gamma$ is the circle with radius 1.\\

\section{Analysis of the abstract reasoning}

Let $A$ be a commutative Banach
algebra with unit and let $\MFn(A)$ be the compact space of
non-zero multiplicative
functionals $A\rightarrow\compls$. Let us assume that an element $f$ of $A$ satisfies
$\varphi(f)\neq 0$
for all $\varphi$ in $\MFn(A)$; and we want to show that $f$ is then invertible
in $A$. To obtain a contradiction, we assume that $f$ is {\em not} invertible. Then the
ideal $\ideal{f}$ is proper and hence, using
Zorn's Lemma, included in a maximal ideal $\gm$. This maximal ideal
is necessarily closed. Using the Gelfand-Mazur theorem, the algebra
$A/\gm$ is isomorphic to $\compls$ and the corresponding multiplicative
functional
$\varphi_{\gm}:A\rightarrow A/{\gm}$ is such that $\varphi_{\gm}(f)=0$. This
contradicts the hypothesis on $f$. For example, the application of this argument
to
$A=l^1(\ints)$ gives Wiener's theorem on inverse of Fourier series.

Cohen~\cite{Cohen} observes that, instead of working with a maximal ideal, one
can work just as well with the closure $I$ of ideal generated by $f$. Moreover,
an ideal contains $1$ if, and only if, its closure contains $1$. This
follows directly from the fact that $x$ is invertible if $|1-x|<1$. Finally, the
proof of the Gelfand-Mazur Theorem actually gives a more general result:
if $A$ is a non trivial Banach algebra, the spectrum of any element $u$
in $A$ is non empty. A combination of these remarks eliminates the use of Zorn's
Lemma.

To simplify, we consider the
case where $A$ is generated by one element $u$: for any $a$ in $A$ and any $r>0$ there is a
polynomial $P$ such that $| a - P(u) | < r$. This covers for instance the
example of the disc algebra~\cite{Cohen}.\label{p:one-gen}
In this case the spectrum $\MFn(A)$ can be identified with the spectrum,
$\sigma(u)$, of $u$: the set of complex $\lambda$ such that $\lambda-u$ is not
invertible. Any such element $\lambda$ defines a multiplicative functional
$\varphi(P(u)):= P(\lambda)$ on polynomials in $u$. If
$|P(\lambda)|>|P(u)|$, then $P(\lambda)-P(u)$ is invertible and hence
$\lambda-u$ is invertible. Since this is not the case, $|P(\lambda)|\leqslant
|P(u)|$, so
$\varphi$ can be extended to a multiplicative functional
$\varphi(g):=g(\lambda)$. Conversely,
for $\varphi$ in $\MFn(A)$, $\varphi(\varphi(u)-u)=\varphi(u)-\varphi(u)=0$ and
hence $\varphi(u)-u$ cannot be invertible, i.e.\ $\varphi(u)\in \sigma(u)$.
The hypothesis on $f$ becomes $f(\lambda)\neq 0$ for all $\lambda$ in the
spectrum of $u$.
The spectrum of $u$ is empty in $A/I$. Indeed, the closure of the ideal
generated by $\lambda-u$ contains $f(\lambda)-f = f(\lambda)$ modulo $I$ which
is invertible and hence $\lambda-u$ is always invertible
modulo $I$. Hence $1$ belongs to $I$ and hence $f$ is invertible. The maximal
ideal of the abstract argument has been replaced by the ``big enough'' ideal
$I$.
Essentially, this is the method followed by de Bruijn and van der
Meiden~\cite{deBruijn}, who
advocated a point-free approach to the description of the spectrum of a Banach
algebra.

In this way we eliminate the use of Zorn's Lemma, but the argument is still
non-constructive.
Cohen shows that we can follow
the proof of Gelfand-Mazur and produce an actual computation of the
inverse of $f$.
Bishop and Bridges' work is
similar~\cite{Bishop/Bridges:1985}.
Coquand and Stolzenberg~\cite{Coquand/Stolzenberg}
show that in most cases, we can obtain a relatively short and explicit formula
for the inverse.

In this paper we suggest a slightly different constructive argument.
It is enough to have a constructive proof of the following result.

\begin{theorem}If we have $u$ in $A$ such that for all $\lambda$ the element $\lambda- u$ is
invertible \emph{ and} furthermore the inverse is uniformly bounded, then $1=0$ in $A$.
\end{theorem}

 Notice that, constructively, we have to state explicitly that the inverse is
uniformely bounded. At first the conclusion seems purely negative, since it says that a
ring is trivial, and not informative. But if we apply the Theorem to a quotient algebra $A/I$ the
conclusion shows that $1$ is in $I$. Choosing $I$ to be the closure of $\ideal{f}$ actually builds
an inverse of $f$, provided we work constructively.
Such use
of ``trivial rings'' occurs often in constructive algebra \cite{Richman,
Seminormal}.
In the applications to Wiener's Theorems on Fourier series and
on the
Tauberian Theorem, we believe that our treatment is simpler than the one in~
\cite{Bishop/Bridges:1985}.
Moreover, this theorem has a rather direct proof, for instance, it does
not rely on Cauchy's formula like the treatment
in~\cite{Cohen,Bishop/Bridges:1985,Coquand/Stolzenberg}.

However, one should notice that this constructive reading does not yet have the
``elegance'' of the classical proofs, since we have to deal explicitely with
a bound of the inverse. In the sequel to this paper, we explain how to ``hide''
explicit mentions to this bound, by using a localic, or pointfree, presentation of the
spectrum. In this way we obtain a rather faithful constructive explanation
of the classical arguments.

\section{Integration and differentiation with values in a Banach space}

We consider a vector space $E$ over the real numbers with an
upper real valued semi-norm: if $a$ is in $E$ then $|a|$ is an
\emph{upper real} (open nonempty upper set of rationals) such that
$|ra|= |r||a|$ and $|a + b| \leqslant |a|+|b|$. This is an important
difference with Bishop's treatment, for whom the norm is a Cauchy
real\footnote{The fact that the norm is not a Cauchy real constructively arises
if one wants to quotient a Banach space by a closed subspace which may not be
located. (This will indeed occur in the applications.)\\
Another motivation is the interpretation of these results in Banach algebra
bundles which we will discuss in Section~\ref{conclusion}. In a Banach algebra
bundle the function $x\mapsto |a_x|$ is \emph{upper} semi continuous, but
generally not continuous; see \cite{Banach-alg-bundle}.\\
Bishop and Bridges~\cite[p.462]{Bishop/Bridges:1985} state:`It would be
interesting, and probably nontrivial, to extend the theory to cover such
algebras [where the norm is not a Cauchy real].'
}.

The equality in $E$ is such that $|a|= 0$ iff $a = 0$ in $E$. If $f : [a, b]
\rightarrow E$ is a (uniformly) continuous map, we say that $f$ is
\emph{differentiable} iff there exists a continuous $f' : \reals
\rightarrow E$ such that for all $\varepsilon > 0$ there exists $\eta$ such that
$|f (y) - f (x) - (y - x) f' (x) | \leqslant \varepsilon |y - x|$ if $|y -
x|
\leqslant \eta$.

\begin{lemma}
  \label{main}If $|f' (x) | \leqslant M$ for all $x$ in $[a, b]$ then $|f
(b)
  - f (a) | \leqslant M (b - a)$. In particular, if $f' (x) = 0$ for all
$x$
  in $[a, b]$ then $f (b) = f (a)$.
\end{lemma}

\begin{proof}
  For any $\varepsilon > 0$ we show $|f (b) - f (a) | \leqslant (M +
  \varepsilon) (b - a)$, by cutting $[a, b]$ in subintervals $[a_i, a_{i +
  1}]$ such that $|f (a_{i + 1}) - f (a_i) - (a_{i + 1} - a_i) f' (a_i) |
  \leqslant \varepsilon (a_{i + 1} - a_i)$.
\end{proof}

Let $E$ be a \tmtextit{Banach space}, i.e.\ a normed vector space which is complete: any
Cauchy approximation converges. If $g : [a, b] \rightarrow E$ is uniformly continuous we define
$\int_a^b g
(x) \mathd x$. If we write $G (y) = \int_a^y g (x) \mathd x$ then $G : [a, b]
\rightarrow E$ is differentiable  and $G' (y) = g (y)$.

\begin{corollary}
If $|g(x)|\leqslant M$ for all $x$ in $[a,b]$, then
$|\int_a^b g(x)\mathd x|\leqslant M(b-a)$.
\end{corollary}

\section{Exponential function}

A (commutative) \tmtextit{Banach algebra} $A$ is a Banach space with a ring
structure, multiplication $ab$ being such that $|ab| \leqslant |a||b|$.
In this section, we assume also that $A$ has a unit element $1$.
\begin{lemma}
  \label{inv}If $|1 - x| < 1$, then $x$ is invertible with inverse
  $\sum (1 - x)^n$.
\end{lemma}

\begin{proof}
  We write $y = 1 - x$. We have $x = 1 - y$ and $\sum y^n$ is the inverse of
  $1 - y$ since $|y| < 1$.
\end{proof}

We define the exponential
$$e^a = \sum \frac{a^n}{n!}$$
of an element $a$ of $A$. The map $x \mapsto e^{ax}$ defines a
function $e_a : \reals \rightarrow A$ such that $e_a$ is differentiable, $e_a'
(x) = ae_a (x)$ and
$e_a (0) = 1$. Furthermore, we have $e_a (x + y) = e_a (x) e_a (y)$.

\begin{lemma}
  \label{exp}Let $f : \reals \rightarrow A$ be a continuous function such that
$f(0) = 1$ and $f (x + y) = f (x) f (y)$. Then $f$ is differentiable and $f =
e_a$
with $a = f' (0)$.
\end{lemma}

\begin{proof}
  We have $1 / t \int_0^t f (x) \mathd x \rightarrow 1$ if $t \rightarrow 0$. So
we
  can find $t > 0$ such that $|1 - 1 / t \int_0^t f (z) dz| < 1$. The element
  $v = 1 / t \int_0^t f (z) dz$ is then invertible by Lemma~\ref{inv}. On the
  other hand we have, using $f (x + y) = f (x) f (y)$
  \[ tf (x) v = \int_0^t f (x + z) dz = \int_x^{x + t} f (z) dz \]
  and hence
  \[ tf (x) = v^{- 1} \int_x^{x + t} f (z) d (z) = v^{- 1} ( \int_0^{t + x} f
     (z) dz - \int_0^x f (x) \mathd x). \]
  It follows that $f$ is differentiable  and
  \[ tf' (x) = v^{- 1} (f (x + t) - f (x)). \]
  Since $f (x + y) = f (x) f (y)$, we have $f' (x + y) = f (x)
  f' (y)$. In particular, if we write $a = f' (0)$ we have $f' (x) = af (x)$.
  If we write $g (x) = f (x) e_a (- x)$ we have $g (0) = 1$ and $g' (x) = af
  (x) e_a (- x) - af (x) e_a (- x) = 0$. Hence by Lemma~\ref{main} we have $g
  (x) = 1$ for all $x$ and so $f (x) = e_a (x)$.
\end{proof}

\section{Path integration}

Let $E$ be a Banach space over of the complex numbers. We say that $f :
\compls \rightarrow E$ is \emph{differentiable} iff there exists a (uniformly)
continuous function $f' : \compls \rightarrow E$ such that for all
$\varepsilon > 0$ there exists $\eta > 0$ such that
\[ |f (z') - f (z) - (z' - z) f' (z) | \leqslant \varepsilon |z' - z| \]
whenever $|z' - z| \leqslant \eta$.

If $\gamma : [0, 1] \rightarrow E$ is a differentiable  function and $f :
\compls \rightarrow E$ is a continuous function, we define $\int_{\gamma}
f (z) dz$ to be $\int_0^1 f (\gamma (t)) \gamma' (t) dt$. We say that $\gamma$
is a \tmtextit{loop} iff $\gamma (0) = \gamma (1)$.

\begin{lemma}
  \label{holom}If $f : \compls \rightarrow E$ is differentiable  and $\gamma :
  [0, 1] \rightarrow \compls$ is a loop then $\int_{\gamma} f (z) dz = 0$
\end{lemma}

\begin{proof}
  We consider $g : [0, 1] \rightarrow E$ defined by
  \[ g (s) = \int_0^1 f (s \gamma (t)) s \gamma' (t) dt.\]
Then $g(1) = \int f$ and $g(0) = 0$.
  It is direct to see that $g$ is differentiable  and that
  \[ g' (s) = \int_0^1 (f (s \gamma (t)) + s \gamma (t) f' (s \gamma (t)))
     \gamma' (t) dt = \int_0^1 h' (t) dt = h (1) - h (0) = 0 \]
  where $h (t) = \gamma (t) f (s \gamma (t))$. By Lemma~\ref{main}, we have
$g(1) = g(0) =0$, hence the result.
\end{proof}

\section{Inverse function}

Let $A$ be a Banach algebra over the complex numbers with a unit element.

\begin{lemma}\label{open}
If $a$ is invertible, with inverse $b$ such that
$|b| \leqslant M$, then $a - u$ is invertible if $|u| \leqslant c / M$ and $c<1$. It inverse is
bounded by
$M/(1-c)$.
\end{lemma}

\begin{proof}
  The element $a - u$ is invertible iff the element $1 - ub$ is. Since $|u| < 1
/ M$ and $|b| < M$, $|ub| \leqslant c$ and hence $1 - ub$ is invertible by
Lemma~\ref{inv}. The inverse is bounded by $(1 - c)^{-1}$. Hence the inverse of $a-u$ is bounded
by $M/(1-c)$.
\end{proof}

It follows from this lemma that the set of invertible elements $A^{\times}$ of
$A$ is an open subset of $A$.


\begin{theorem}\label{mainC}
 Let $u$ be in $A$. If for all $z$ in $\compls$ the element $z -
u$ is
  invertible, with inverse $f (z)$, \tmtextit{and} $f$ is uniformly bounded,
  then $1 = 0$ in $A$.
\end{theorem}

Notice that classically, the hypothesis that $f$ is bounded is not
needed\footnote{Classically, this theorem implies the Gelfand-Mazur
Theorem: a
Banach algebra $A$ which is also a field is isomorphic to $\compls$.
Indeed, from Theorem~\ref{mainC}, for each $u$ in $A$, there exists $\lambda$
such that
$u-\lambda$ is not invertible and this implies $u=\lambda$.}.

\begin{proof}
  We have $f (z') - f (z) = (z - z') f (z) f (z')$ and hence $f$ is Lipschitz
  continuous and differentiable with $f' (z) = - f (z)^2$. We consider $\gamma :
  [0, 1] \rightarrow \compls$ to be the circle loop $\Gamma_R (t) = Re^{2
\pi it}$
  for some $R>|u|$. By Lemma~\ref{holom} we have
  \[ \int_{\Gamma_R} f (z) dz = 0.\]
  On the other hand, for $R$ big enough $f (z)$ is equal to $\sum u^n / z^{n
  + 1}$ over $\Gamma_R$ and so $\int_{\Gamma_R} f (z) dz$ is equal to
  $\int_{\Gamma_R} dz / z = 2 \pi i$. So, we have $1 = 0$ in $A$.
\end{proof}

Before giving the applications we remark that if $I$ is an ideal of a Banach
algebra $A$, we can define a new Banach algebra $A / I$ by taking the new norm
to be: $|a|_I < r$ iff there exists $b$ in $I$ such that $|a - b| < r$. Then
$a = 0$ in $A / I$ iff $a$ belongs to the closure of $I$. It follows
from Lemma~\ref{inv} that $1 = 0$ in $A / I$ iff $1$ belongs to $I$.

We deduce the following method from the previous theorem: To prove that an
element $g$ of $A$ is invertible, it is enough to find $u$ in $A$ such that for all $z$ in $\compls$
the element $z - u$ is invertible of bounded inverse  in $A / \ideal{g}$.

We emphasize that classically the inverse is always uniformly bounded. Constructively, there is a
meta-theorem, the fan rule, that allows us to always find a bound in concrete case; see
e.g.~\cite{Troelstra/vanDalen:1988}. In the examples below we make an effort to be explicit about
the bound.

\section{Application 1: Wiener's Theorem on Fourier series}

We present Wiener's theorem on Fourier series; see e.g.~\cite{Loomis}. The Banach algebra $B =
l^1(\ints)$ is the completion of the algebra of sequences in $\compls^{\ints}$ of finite support
with the convolution product $(a*b)_n = \sum a_ib_{n-i}$ and the norm $|a| = \sum |a_n|$. This
algebra has unit $\delta_0$. We write $u$ for $\delta_1$. Then $u^{-1}=\delta_{-1}$ and for every
$a$,
$a = \sum a_n u^n$ and $B$ is simply an algebra of infinite series under formal multiplication.

 If $\lambda$ is in $\Gamma$, i.e.\ if $|\lambda| = 1$, then we define
$a(\lambda) = \sum a_n\lambda^n$. We have $|a(\lambda)|\leqslant |a|$.

\begin{theorem}
If $f$ is an element of $B$ such that $|f(\lambda)|\geqslant\varepsilon$
for all $\lambda$ in $\Gamma$, then $f$ is invertible.
\end{theorem}

\begin{proof}
We let $A$ be the quotient $B/\langle f\rangle$. We show that $\lambda -u$ is
invertible for all $\lambda$ in $\compls$ and furthermore that the inverse is uniformly
bounded. We can then apply Theorem~\ref{mainC}.

 If $|\lambda|<1$, then $|\lambda u^{-1}|=|\lambda|\leqslant r$ for some $r<1$. Hence, $\sum_{n\geq
1}
(\lambda/u)^n$ exists and equals $(1-\lambda u^{-1})^{-1}=u(u-\lambda)^{-1}$. The inverse
$(u-\lambda)^{-1}$ is bounded by $|u^{-1}| (1-r)^{-1}=(1-r)^{-1}$.

Similarly, if $|\lambda|>1$, then $|u/\lambda|= |1/\lambda| \leqslant r$ for some $r<1$. Hence,
$\sum_{n\geq 1}
(u/\lambda)^n$ exists and equals $(1-(u/\lambda))^{-1}=u^{-1}(\lambda-u)^{-1}$. The inverse
$(\lambda-u)^{-1}$ is bounded by $|u| (1-r)^{-1}=(1-r)^{-1}$.

 Let $g$ be an element of finite support such that $|f-g|<\varepsilon'$. Then
\[|(g(\lambda)-g)- (f(\lambda)-f)|=|(g-f)(\lambda)+(f-g)|\leqslant 2\varepsilon'.\]
Since $f(\lambda)-f$ is equal to $f(\lambda)$ mod $f$, $g(\lambda)-g$ is invertible mod $f$.
By Lemma~\ref{open} we have a uniform bound $M$ on the inverse of $g(\lambda)-g$ in $\Gamma$. Since
$g(\lambda)-g$ is a polynomial in $u$, $\lambda -u$ divides it, say
$(\lambda -u)h=g(\lambda)-g$. So, $\lambda-u$ is invertible for all $\lambda$ in $\Gamma$ and the
inverse is bounded by $M|h|$. We assume that $M|h|\geqslant 1$. By Lemma~\ref{open}, $\lambda-u$ is
invertible for all $\lambda$ with $||\lambda|-1|\leqslant1/(2M|h|)$ and its inverse is bounded by
$|Mh|/(1-\frac12)=2|Mh|$.

Write $\alpha:=1/(2M|h|)$. Then either $||\lambda|-1|\leqslant\alpha$ or
$||\lambda|-1|\geq \alpha/2$.
In the latter case, either $|\lambda|\leqslant1-\alpha/2$ or $|\lambda|\geq 1+\alpha/2$.
We conclude that $\lambda-u$ is invertible for all $\lambda$ in $\compls$ and its inverse is bounded
by \[\sup \{2M|h|,(1-\frac\alpha2)^{-1},(1-(1+\frac\alpha2)^{-1})^{-1}\}.\]
\end{proof}

 Being constructive, this reasoning can be seen as an algorithm that computes
an inverse of $f$ in $B$.

\section{Application 2: Wiener's Tauberian Theorem\label{Tauberian}}

Let $C(\reals)$ be the space of continuous functions of compact support
and let $L = L^1(\reals)$ be the completion of this space
for the $L^1$ norm. If we define $T_x(g)(y) = g(x+y)$, for $g$ in $C(\reals)$,
we have $|T_x(g)-T_x(h)| = |g-h|$ and, by extension,
we have a continuous function $x\mapsto T_x(g),~\reals\rightarrow L^1(\reals)$
for any $g$ in $L$ such that $|T_x(g)| = |g|$ for all $x$.
For $f$ in $C(\reals)$ and $g$ in $L$, we define
$$
f*g = \int f(x)T_{-x}(g) \mathd x.
$$
It follows that we have $|f*g|\leqslant |f|_\infty|g|_1$. Hence the product $f*g$
can be
defined by extension for $f$ and $g$ both in $L$.
If both $f,g$ are in $C(\reals)$, so is $f*g$.

Since we have
$$(f*g)(y)=(g*f)(y) = \int f(x)g(y-x) \mathd x$$
for $f,g$ in $C(\reals)$ this holds also for $f,g$ in $L$ and
the algebra $L$ with convolution product is a commutative Banach algebra. (It
can be shown that it does not have any unit element.)

 If $g$ is a continuous function with compact support we define
$$\hat{g}(p) = \int g(x)e^{-ipx} \mathd x.$$
We then have $\hat{g}$ in $C_0(\reals)$, that is, $\hat{g}$ is continuous on
$\reals$ and tends to $0$ at infinity. This is first proved for $g\in C^1(\reals)$ by integration
by parts and then for
general $g\in C(\reals)$ by density.
Moreover, $|\hat{g}(p)|\leqslant
|g|$ for
all $p$. It follows that we can extend the map $g\mapsto \hat{g}$ and define
$\hat{g}$ in $C_0(\reals)$ for $g$ in $L$.

 We fix an element $h$ of $L$ such that $\hat{h}(p) = 1$ for all $p$ in
$[-M,M]$ and for all $\varepsilon>0$ there exists $\eta$ such that
$|\hat{h}(p)|\leqslant 1-\eta$ if $M+\varepsilon\leqslant|p|$.

\begin{theorem}\label{Tauberian1}
If $|\hat{f}(p)|>\varepsilon$ for all $p$ in $[-M,M]$, then $f$ divides any
element $g$ such that $g * h = g$ in $L$.
\end{theorem}

\begin{proof}
We let $I$ be the ideal of all elements $k*h - k$ and $B$ be the Banach
algebra $L/I$. Notice that $B$ has $h$ as unit element.
We claim that $f$ is invertible in $B$.

 The map $\alpha:\reals\rightarrow B$ defined by $\alpha(x) = T_x(h)$ is
continuous
and satisfies $\alpha(0) = h$ and $\alpha(x+y) = \alpha(x) * \alpha(y)$. By
Lemma~\ref{exp} there exists an element $u$ in $B$ such that $T_x(h) =
e^{ux}$ for
all $x$ in $\reals$. In $B$ we have
\[g = g*h = \int g(x)T_{-x}(h) \mathd x = \int g(x)e^{-ux} \mathd x.\]
We show that $u-\lambda$ is invertible mod $f$ and of bounded inverse. We can then
apply Theorem~\ref{mainC} and deduce that $f$ is invertible in $B$.

We will consider three cases: 1.\ $\lambda=ip$ for $|p|\leqslant M$, 2.\ $\lambda=ip$ for
$|p|\geqslant M+\delta$ and $\delta\geqslant 0$, and 3.\ $\lambda=r+ip$ and $|r|>\delta>0$. By
continuity and
Lemma~\ref{open} these cases are sufficient.

1.\ We claim that $ip-u$ is invertible modulo $f$ if $|p|\leqslant M$. Indeed,
if we take $N$ such that
\[\left| \int_{-N}^N |f(x)|\mathd x - \int |f(x)|\mathd
x\right|\leqslant\frac\varepsilon{2(1+|h|)},\]
then
\begin{eqnarray*}
 \left| \int f(x)(e^{-ipx} - e^{-ux})\mathd x- \int_{-N}^N f(x)(e^{-ipx} -
e^{-ux})\mathd x\right|&\leqslant\\
\left| \int_{\reals\setminus[-N,N]} f(x)(e^{-ipx} - e^{-ux})\mathd x\right|&\leqslant\\
\left| \int_{\reals\setminus[-N,N]} f(x)\right| (1+|h|)\mathd
x&\leqslant&\frac\varepsilon2.\\
\end{eqnarray*}
Since,
\[\int f(x)(e^{-ipx} - e^{-ux})\mathd x = \hat{f}(p)-f.\]
this element is invertible mod $f$, by hypothesis its inverse is bounded by $1/\varepsilon$.
By Lemma~\ref{open}, $\int _{-N}^N f(x)(e^{-ipx} - e^{-ux})\mathd x$ is invertible mod $f$. Let
$g$ be its inverse. Then $|g|\leqslant 1/\varepsilon(1-\frac12)=2/\varepsilon$. Since this integral
is divisible by $ip-u$, it
remains to find an explicit bound for the inverse. We have $e^{-ipx} - e^{-ux}=e^{-ipx}(1 -
e^{(ip-u)x})$ and for all $r$, $r|(1-e^{rx})$ and $|\frac{1-e^{rx}}r|\leqslant 1+e^{|r|x}$, as a
simple
power series argument shows.
Consequently, the inverse of $ip-u$ is bounded by $\frac2\epsilon\int _{-N}^N
|f(x)|(1+e^{rx})\mathd x$, where $r>|ip-u|$.

2.\ We claim that $ip-u$ is invertible with bounded inverse for all $p$ such that
$|p|\geqslant M+\delta$.
Indeed, if we take
$\eta>0$ such that $|\hat{h}(p)|\leqslant 1-\eta$ and $N$ such that
\[\left|\int _{-N}^N |h(x)|\mathd x-\int |h(x)| \mathd x\right|\leqslant\eta/2(1+|h|).\]
then
\begin{eqnarray*}
\left|\int _{-N}^N h(x)(e^{-ipx} - e^{-ux})\mathd x - \int h(x)(e^{-ipx} - e^{-ux})\mathd
x\right|&\leqslant\\
\left| \int_{\reals\setminus[-N,N]} h(x)(e^{-ipx} - e^{-ux})\mathd x\right|&\leqslant\\
\left| \int_{\reals\setminus[-N,N]} h(x)\right| (1+|h|)\mathd x&\leqslant&\eta/2\\
\end{eqnarray*}
and
\[\left|\int h(x)(e^{-ipx} - e^{-ux})\mathd x\right|= \hat{h}(p)-h\]
the latter term is invertible since $h$ is the unit of $B$.
Hence $\int _{-N}^N h(x)(e^{-ipx} - e^{-ux})\mathd x$ is invertible with inverse bounded by
$2/\eta$. As before, this element is divisible by $ip-u$ and its inverse is bounded by
$\frac2\eta\int _{-N}^N |h(x)||1+e^{rx}|\mathd x$, where $r>|ip-u|$.

We conclude that $ip-u$ is invertible for all $p$ in $\reals$ and its inverse is bounded by
$d:=\frac2\epsilon\int (|h(x)|+|f(x)|)(1+e^{rx})\mathd x$, where $r>|ip-u|$.
Using Lemma~\ref{open} it follows that we can find $\delta>0$
such that $r+ip-u$ is invertible for all $p$ in $\reals$ and $r$ such that
$|r|<1/2d$, the inverse is bounded by $2d$.

3.\ For $\lambda=r+ip$, $\lambda -u$ divides

\begin{equation}
1-e^{(\lambda - u)x} = 1- e^{\lambda x} T_x(h) = (1-e^{-rx}e^{-ipx}T_x(h))\label{form}
\end{equation}

and $|T_x(h)|$ is bounded by $|h|$. For $x=\frac{\log(2|h|)}r$,
$|e^{-rx}e^{-ipx}T_x(h)|\leqslant e^{-rx}|h|= \frac12$, so the right hand side of
equation~\ref{form} is invertible with inverse bounded by $2e^{-r x}=2/(2|h|)=1/|h|$. Hence $\lambda
-u$ has an inverse with bound $1/|h|$ for $|r|>\varepsilon$.

We have finished the proof that for each $\lambda$, $u-\lambda$ is invertible mod $f$ and of bounded
inverse.

 It follows that we have $k$ such that $f*k = h$ mod $I$. If we assume $g * h =
g$
we have $g * I = 0$ and then $f * k * g = g$ so that $f$ divides $g$.
\end{proof}

\begin{proposition}
  {\emph{{\cite[VI.1.13]{Katznelson}}}} If $g$ in $L^1 (\reals)$ is such
  that $\hat{g}$ has compact support then there exists $h$ in $L^1(\reals)$
such
  that $h \ast g = g$. In fact, we can use a De la Vall\'ee Poussin's kernel
  for $h$.
\end{proposition}
We deduce the following version of Wiener's Tauberian Theorem.

\begin{corollary}\label{ideal-dense}
If $f$ in $L$ is such that $\hat{f}$ never takes the value $0$, then every
function $g$ such that $\hat{g}$ is of compact support is in the ideal
generated by $f$. Consequently, this ideal is dense in $L$.
\end{corollary}
\begin{proof}
  Any function in $L^1 (\reals)$ is the limit of functions in $L^1
  (\reals)$ having their Fourier transform of compact
support~\cite[VI.1.12]{Katznelson}:

Define the Fej\'er kernel $K_\lambda(x)=\lambda K(\lambda x)$, where $K(x)=\frac1{2\pi}(\frac{\sin
x/2}{x/2})^2$.

By the inversion formula:
\[
f(x)=\frac1{2\pi}\int\hat{f}(t)e^{itx}\mathd t,
\]
$\hat{K_\lambda}(t)=\max(1-|t|/\lambda,0)$ and one can derive that
$(\widehat{K_\lambda*f})(t)=(1-|t|/\lambda)\hat{f}(t)$ for $|t|\leqslant\lambda$ and 0 otherwise.
\end{proof}

 Since $f * g = \int g(x)T_{-x}(f) \mathd x$ another way to state this Corollary
is
the following one.

\begin{corollary}
If $f$ in $L$ is such $\hat{f}$ never takes the value $0$ then the vector space
generated by the functions $T_x(f)$ is dense in $L$.
\end{corollary}

Constructively, the hypothesis that $\hat{f}$ never takes the value $0$ should
be read as:
$\hat{f}$ is bounded away from $0$ on any compact.

\section{Conclusions and Future work}\label{conclusion}

To have a point-free description of the spectrum of a Banach algebra was
the goal of the work of de Bruijn and van der Meiden~\cite{deBruijn}. The
complexity of this description and the process of finding the proof
of the compactness of the spectrum is cited by de Bruijn as one inspiration for his
AUTOMATH project~{\cite{deBruijn,deBruijn:AUTOMATH}}.

 It would be interesting to have an actual implementation of our work, following already existing
work formalising basic analysis~\cite{Riemann}.

 The work of Krivine~\cite{Krivine} contains several examples similar to Wiener's
results, but in a `real' framework. One considers respectively $l^1(\nats)$
instead of $l^1(\ints)$ and $L^1(\reals^+)$ instead of $L^1(\reals)$. It might
be interesting to give a constructive interpretation these results using the
technique presented in~\cite{Coquand:Stone,StoneYosida}.

Constructive, and choice-free, results on Banach spaces can be interpreted as results on Banach
sheaves, or equivalently, Banach bundles~\cite{Banach-sheaves}. Similarly,
constructive results on Banach algebras can be interpreted as results on Banach
algebra bundles~\cite{Banach-alg-bundle}.

\bibliographystyle{alpha}\bibliography{B-alg}

\end{document}